\documentclass[12pt]{amsart}
\usepackage{amsfonts}
\usepackage{amsthm}
\usepackage{amssymb}
\usepackage{amsmath}
\usepackage{amsmath,amscd}
\usepackage{epsfig}
\usepackage{url}
\usepackage{color}
\usepackage[margin=1.1in]{geometry}
\usepackage{setspace}
\usepackage{graphicx}%
\providecommand{\U}[1]{\protect\rule{.1in}{.1in}}

\newtheorem{theorem}{Theorem}
\newtheorem*{theorem*}{Theorem}
\newtheorem{corollary}{Corollary}
\newtheorem{lemma}{Lemma}
\newtheorem{proposition}{Proposition}
\newtheorem{definition}{Definition}

\newcommand{\R}{\mathbb{R}}

\def\p{\partial}
\def\R{\mathbb{R}}

\def\Z{\mathbb{Z}}

\begin{document}
\title{Legendrian curve shortening flow in $\mathbb{R}^{3}$}

\author{Gregory Drugan}
\address{Department of Mathematics, University of Oregon, Eugene OR 97403}
\email{drugan@uoregon.edu}
\author{Weiyong He}
\address{Department of Mathematics, University of Oregon, Eugene OR 97403}
\email{whe@oregon.eduu}
\thanks{The second author was partially supported by NSF Grant DMS-1005392.}
\author{Micah W. Warren}
\address{Department of Mathematics, University of Oregon, Eugene OR 97403}
\email{micahw@oregon.edu}
\thanks{The third author was partially supported by NSF Grant DMS-1438359. }
\maketitle

\begin{abstract}
Motivated by Legendrian curve shortening flows in $\mathbb{R}^{3}$, we study the curve shortening flow of figure-eight curves in the plane. We show that, under some symmetry and curvature conditions, a figure-eight curve will shrink to a point at the first singular time.
\end{abstract}

\pagestyle{empty}

\bigskip

\section{Introduction}

We show the following theorem.

\begin{theorem}
Suppose that $\gamma(u)$ is a smoothly immersed figure-eight shape in $\mathbb{R}^{2}$ that encloses zero signed area, is symmetric about an interior axis, and has exactly two inflection points. Then the curve shortening flow collapses $\gamma$ to a point at the first singular time.
\end{theorem}

This gives a partial answer to a conjecture of Grayson \cite{G2}, which states that all figure-eight curves with zero signed area should shrink to a point under curve shortening flow.  In particular, our main result provides a class of curves that do in fact shrink to a point.

These figure-eight shapes arise naturally in the study of Legendrian curve shortening flow, as defined by Smoczyk \cite{Smoczyk}. Combining this observation with our main theorem, we obtain the following result.

\begin{corollary}
With the standard contact structure $\eta=dz-ydx$ on $\mathbb{R}^{3}$, there exist embedded Legendrian immersions $\gamma: S^{1} \rightarrow \mathbb{R}^{3}$ such that a Legendrian curve shortening flow shrinks $\gamma$ to a point at the first singular time. 
\end{corollary}

Because mean curvature flow does not preserve the Legendrian condition, one has to modify the flow to obtain a flow of Legendrian immersions.  A natural approach is to put a metric on the space of Legendrian immersions and find the negative gradient flow for the length functional (or the area funtional, for higher dimensions).  L\^{e} \cite{Le} has suggested a gradient flow which produces a fourth order equation. In the curve shortening case, this is the curve diffusion flow, which has been studied in \cite{EGMWW}. In \cite{EGMWW}, the authors show that the lemniscate of Bernoulli shrinks to a point in finite time under the gradient flow defined by L\^{e}.  There are two other natural candidates for metrics. One metric produces a nonlinear system, which to our knowledge has not been studied. The other is an indefinite metric,  which measures the deformation in the Legendrian normal direction.  Curves of maximal slope with respect to this metric (in the sense of \cite{AGS}) are precisely the Legendrian curve shortening flows.  These are non-unique, but they do have the property that they always project to curve shortening flow on the base space.  For more discussion of the gradient flows associated to Legendrian curve shortening flow, see Section \ref{GF}.

The proof of our main result is described heuristically as follows. Under the symmetry condition, the curve shortening flow forces the figure-eight to move towards an axis. Using volume and convexity considerations, one can then carefully place symmetric grim reaper curves around the figure-eight. Maximum principle arguments show that if the curve is not collapsing fast enough, the grim reapers must be moving with enough speed to push the curve to infinity, which is a contradiction. In fact, this argument gives us a rate of collapse. 

\bigskip

\section{Legendrian curve shortening flow}

Consider the contact manifold ${\mathbb{R}}^{3}$ with coordinates $(x,y,z)$ and contact form: $\eta= dz -y dx$. Let $g = dx \otimes dx + dy \otimes dy +\eta\otimes\eta$ be the associated Riemannian metric, and let $\xi= \partial_{z}$ denote the Reeb vector field.

We define the vector fields $X=\partial_{x}+y\partial_{z}$, $Y=\partial_{y}$, and $Z=\partial_{z}$ (Reeb vector field), which form an orthonormal basis for ${\mathbb{R}}^{3}$ with respect to $g$. We have the relations:
\[
\lbrack X,Y]=-Z,\quad\lbrack X,Z]=[Y,Z]=0.
\]
In terms of the Levi-Civita connection $\nabla$, we record the identities:
\begin{equation}
\nabla_{X}X=\nabla_{Y}Y=\nabla_{Z}Z=0,\quad\nabla_{X}Y=-\nabla_{Y}X=-Z/2,
\label{cov}
\end{equation}
\[
\nabla_{X}Z=\nabla_{Z}X=Y/2,\quad\nabla_{Y}Z=\nabla_{Z}Y=-X/2.
\]

\begin{definition}
A $C^{1}$ regular curve $\gamma:S^{1}\rightarrow{\mathbb{R}}^{3}$ is called \textbf{Legendrian} if
\begin{equation}
\gamma^{\ast}(\eta)=0. \label{defLeg1}
\end{equation}
\end{definition}

If $\gamma:S^{1}\rightarrow{\mathbb{R}}^{3}$ is a Legendrian curve with coordinate $u$: $\gamma=\left(  x(u),y(u),z(u)\right)$, then condition (\ref{defLeg1}) becomes
\begin{equation}
z_{u}-yx_{u}=0. \label{defLeg2}
\end{equation}
It follows that $|\gamma_{u}|_{{}_{g}}=\sqrt{x_{u}^{2}+y_{u}^{2}}$, and the unit tangent vector $T=\gamma_{u}/|\gamma_{u}|_{{}_{g}}$ is given by
\[
T=\frac{x_{u}X+y_{u}Y}{\sqrt{x_{u}^{2}+y_{u}^{2}}}.
\]
Defining a vector field $N$ along $\gamma$ by
\[
N=\frac{-y_{u}X+x_{u}Y}{\sqrt{x_{u}^{2}+y_{u}^{2}}},
\]
we have an orthonormal basis $\{T,N,\xi|_{\gamma}\}$ for ${\mathbb{R}}^{3}$ along $\gamma$.

\begin{lemma}
\[
\nabla_{T}T=\kappa N,\quad\text{where}\quad\kappa=\frac{x_{u}y_{uu} -y_{u}x_{uu}}{(x_{u}^{2}+y_{u}^{2})^{3/2}}.
\]
\end{lemma}

\begin{proof}
\begin{equation}
\begin{aligned}
\nabla_T T & = \frac{ \frac{d}{du} ( \frac{x_u}{\sqrt{x_u^2 + y_u^2}} ) X  + \frac{d}{du} ( \frac{y_u}{\sqrt{x_u^2 + y_u^2}} ) Y }{ \sqrt{x_u^2 + y_u^2} }  \,  + \,  \frac{x_u \nabla_T X + y_u \nabla_T Y}{\sqrt{x_u^2 + y_u^2}} \nonumber \\ 
& = \frac{ x_u y_{uu} - y_u x_{uu} }{ (x_u^2 + y_u^2)^{3/2} } N  \,+  \,  0 \nonumber \\ 
& = \kappa N, \nonumber 
\end{aligned}
\end{equation}
where $\kappa=\frac{x_{u}y_{uu}-y_{u}x_{uu}}{(x_{u}^{2}+y_{u}^{2})^{3/2}}$ and we have used the identities (\ref{cov}).
\end{proof}

\smallskip

\begin{definition}
For a Legendrian curve $\gamma(u)$, we say that $\lambda(u): S^{1}\rightarrow S^{1}$ is a \textbf{Legendrian angle function} if
\[
d\lambda\left(  \frac{\partial_{u}}{\sqrt{x_{u}^{2}+y_{u}^{2}}}\right) =\kappa(\gamma(u))\partial_{\theta},
\]
where $u$ and $\theta$ are coordinates on $S^{1}.$
\end{definition}

It follows that if a Legendrian curve $\gamma(u): S^{1}\rightarrow {\mathbb{R}}^{3}$ satisfies
\[
\int_{S^{1}}\kappa\sqrt{x_{u}^{2}+y_{u}^{2}}du=0,
\]
then any Legendrian angle $\lambda(u)$ can be lifted to a single-valued function, which we use interchangeably with a corresponding real valued function $\operatorname{Im}\log\lambda(u)$.

\begin{definition}
A family of Legendrian curves $\gamma(t,u):I\times S^{1}\rightarrow{\mathbb{R}}^{3}$ satisfying $\int_{\gamma}\kappa=0$ is a solution to the \textbf{Legendrian curve shortening flow} if
\begin{equation}
\gamma_{t}=\kappa N+\lambda\xi|_{\gamma},\label{lcsf}
\end{equation}
where $\lambda(t,u)$ is a family of (real, single-valued) Legendrian angle functions.
\end{definition}

Now as we have not specified $\lambda$, we note that given an initial datum $\gamma_{0}:S^{1}\rightarrow{\mathbb{R}}^{3}$ there are many possible choices of Legendrian angle function. In fact, the family of curves $\left\{  \gamma-(0,0,f(t)) \right\}  $ is a solution to the flow for any choice of function $f(t)$. If we would like to specify a unique flow, we can simply insist that, say, $z(t,0)=z(0,0)$. With this in mind, we say that $\gamma(t,u):I\times S^{1}\rightarrow{\mathbb{R}}^{3}$ is a solution to the \textbf{normalized Legendrian curve shortening flow} if
\begin{align}
\gamma_{t} &  =\kappa N+\lambda\xi|_{\gamma}, \label{nlcsf} \\
z(t,0) &  =z(0,0).\label{nlcsf2}
\end{align}

The following result provides a structure for flows of Legendrian curves.
\begin{lemma}
\label{Lemma1} Suppose that $\gamma_{0}:S^{1}\rightarrow{\mathbb{R}}^{3}$ is Legendrian and the family of curves $\gamma(t,u):I\times S^{1}\rightarrow{\mathbb{R}}$ satisfies
\[
\gamma_{t}=\frac{1}{\left\vert \partial u\right\vert _{g}}f_{u}N+f\xi
|_{\gamma},
\]
with initial datum $\gamma_0$. Then $\gamma\left(t,\cdot\right)$ is Legendrian for all $t \in I$.
\end{lemma}

\begin{proof}
Let $\gamma(t,u)=(x(t,u),y(t,u),z(t,u))$ be a family of curves, and suppose that the normal variation satisfies
\[
\gamma_{t}= \phi N+ f \xi|_{\gamma}.
\]
Then
\begin{align*}
\gamma_{t}  & = \phi \left(  \frac{-y_{u} X + x_{u} Y}{\sqrt{x_{u}^{2}+y_{u}^{2}}} \right) + f \partial_{z} \\
& = -\phi \frac{y_{u}}{\sqrt{x_{u}^{2}+y_{u}^{2}}} \partial_{x} + \phi \frac{x_{u}}{\sqrt{x_{u}^{2}+y_{u}^{2}}} \partial_{y} + \left( f - \phi y \frac{y_{u}}{\sqrt{x_{u}^{2}+y_{u}^{2}}} \right) \partial_{z}.
\end{align*}
We compute
\begin{align*}
\frac{\partial}{\partial t}\gamma^{\ast}(\eta)  & =\frac{\partial}{\partial t} \left( z_{u}-yx_{u} \right)  \\
& =\frac{\partial}{\partial u}\frac{\partial z}{\partial t}-\frac{\partial y}{\partial t}\frac{\partial x}{\partial u}-y\frac{\partial}{\partial u} \frac{\partial x}{\partial t} \\
& =\frac{\partial}{\partial u}\left(f - \phi y \frac{y_{u}}{\sqrt{x_{u}^{2}+y_{u}^{2}}} \right)  - \left( \phi \frac{ x_{u} }{\sqrt{x_{u}^{2}+y_{u}^{2}}} \right) \frac{\partial x}{\partial u}  + y \frac{\partial}{\partial u} \left( \phi \frac{y_{u}}{\sqrt{x_{u}^{2}+y_{u}^{2}}} \right)  \\
& = \frac{\partial}{\partial u} f  - \phi \frac{y_{u}^2}{\sqrt{x_{u}^{2}+y_{u}^{2}}} -  y \frac{\partial}{\partial u} \left( \phi \frac{ y_{u} }{\sqrt{x_{u}^{2}+y_{u}^{2}}} \right)  - \phi \frac{ x_{u}^2 }{\sqrt{x_{u}^{2}+y_{u}^{2}}} + y \frac{\partial}{\partial u} \left( \phi \frac{y_{u}}{\sqrt{x_{u}^{2}+y_{u}^{2}}} \right)  \\
& = \frac{\partial}{\partial u} f - \phi \sqrt{x_{u}^{2}+y_{u}^{2}}.
\end{align*}
Thus, the Legendrian condition is preserved if and only if
\[
\phi = \frac{1}{|\partial u|_g}  f_u.
\]
\end{proof}

The following result shows how Legendrian curve shortening flow in ${\mathbb{R}}^{3}$ projects onto curve shortening flow in ${\mathbb{R}}^{2}$. 

\begin{proposition}
Suppose that $\gamma(t,u):I\times S^{1}\rightarrow{\mathbb{R}}^{3}$ is a solution to the Legendrian curve shortening flow. Then the projection
\[
\bar{\gamma}(t,u):=(x(t,u),y(t,u))
\]
is an immersed solution to the curve shortening flow in ${\mathbb{R}}^{2}.$
\end{proposition}

\begin{proof}
If $\gamma$ is a solution to (\ref{lcsf}), then
\begin{equation}
x_{t}=\kappa\frac{-y_{u}}{\sqrt{x_{u}^{2}+y_{u}^{2}}},\qquad y_{t}=\kappa \frac{x_{u}}{\sqrt{x_{u}^{2}+y_{u}^{2}}}. \label{projflow}
\end{equation}
As $\kappa$ is the expression for curvature in ${\mathbb{R}}^{2}$, this is precisely curve shortening flow.
\end{proof}

With this observation, we have the following existence and uniqueness result for (\ref{nlcsf})(\ref{nlcsf2}).

\begin{theorem}
\label{lcsf:exists} Given a Legendrian curve $\gamma_{0}:S^{1}\rightarrow {\mathbb{R}}^{3}$ such that $\kappa_{0}$ is H\"{o}lder continuous and $\int_{\gamma_{0}}\kappa_{0}=0$, there exists a unique maximal solution $\gamma:[0,t_{max})\times S^{1}\rightarrow\mathbb{R}$ to the normalized Legendrian curve shortening flow starting at $\gamma_{0}$, which exists as long as a H\"{o}lder norm of $\kappa(t)$ remains bounded. Moreover, we have the following properties:

\begin{enumerate}
\item[(i.)] The lifespan $t_{max}$ of the solution is finite, and $\kappa(t)$ becomes unbounded as $t\rightarrow t_{max}$.

\item[(ii.)] For $t>0$, the curve $\gamma(t)$ is real analytic.

\item[(iii.)] As $t\rightarrow t_{max}$, the curves $\gamma(t)$ converge uniformly to a continuous map $\gamma^{\ast}:S^{1}\rightarrow\mathbb{R}^{3}$, which is a piecewise $C^{1}$ curve with a finite number of singular points. Away from these singular points, the limit curve $\gamma^{\ast}$ is real analytic.
\end{enumerate}
\end{theorem}

\begin{proof}
The proof of this theorem follows from identifying solutions to the Legendrian curve shortening flow in $\mathbb{R}^{3}$ with solutions to the curve shortening flow in $\mathbb{R}^{2}$. 

Let $\gamma_{0}:S^{1}\rightarrow {\mathbb{R}}^{3}$ be a Legendrian curve with with coordinate $u$: $\gamma
_{0}=\left(  x_{0}(u),y_{0}(u),z_{0}(u)\right)  $ such that $\kappa_{0}$ is H\"{o}lder continuous and $\int_{\gamma_{0}}\kappa_{0}=0$, where
\[
\kappa_{0}=\frac{(x_{0})_{u}(y_{0})_{uu}-(y_{0})_{u}(x_{0})_{uu}}{((x_{0})_{u}^{2}+(y+0)_{u}^{2})^{3/2}}.
\]
Consider the projection of $\gamma_{0}$ onto its first two coordinates:
\[
\bar{\gamma}_{0}(u)=(x_{0}(u),y_{0}(u)),
\]
where we use the bar notation to denote projection onto the first two coordinates. The curvature vector of $\bar{\gamma}_{0}(u)$ in $\mathbb{R}^{2}$ is given by
\[
\overrightarrow{\kappa_0} = \kappa_{0} \mathbf{n},
\]
where
\[
\mathbf{n}=\frac{(-(y_{0})_{u},(x_{0})_{u})}{\sqrt{(x_{0})_{u}^{2}+(y_0)_{u}^{2}}}.
\]
We let $\bar{\gamma}$ denote the unique maximal solution to the curve shortening flow starting at $\bar{\gamma}_{0}$ (see Angenent \cite{A22} \cite{A2}). Then $\bar{\gamma}$ has all the properties stated in the theorem. 

Differentiating the integral $\int_{\bar{\gamma}}\kappa$, we have
\begin{equation}
\begin{aligned} \frac{d}{dt} \int_0^{2 \pi} \kappa \sqrt{x_u^2 + y_u^2} du & = \int_0^{2 \pi} \left[ \left(\kappa^3 + \frac{\partial ^2 \kappa}{\partial s^2} \right) \sqrt{x_u^2 + y_u^2} + \kappa \left( -\kappa^2 \sqrt{x_u^2 + y_u^2} \right) \right] du \\ 
& = \int_0^{2 \pi} \frac{\partial ^2 \kappa}{\partial s^2} \sqrt{x_u^2 + y_u^2} \, du \\ 
& = \int_0^{2 \pi} \frac{\partial}{\partial u} \left( \frac{\partial \kappa}{\partial s} \right) du \, = \, 0. \end{aligned}\nonumber
\end{equation}
Since, the integral $\int_{\bar{\gamma_{0}}}\kappa_{0}$ was assumed to be zero, we see that $\int_{\bar{\gamma}}\kappa=0$ for all $t\geq0$. 

We observe that (\ref{projflow}) implies
\begin{align}
(yx_{u})_{t}  &  =\left(  yx_{t}\right)  _{u}+y_{t}x_{u}-y_{u}x_{t} \label{forparts1} \\
&  = (  yx_{t} )_{u}+\kappa\sqrt{x_{u}^{2}+y_{u}^{2}}. \nonumber
\end{align}
Differentiating the integral for enclosed signed area $A=-\int_{0}^{2\pi}yx_{u}du$, we have
\begin{equation}
\begin{aligned} 
\frac{d}{dt} A & = - \int_0^{2 \pi} (y x_{u})_{t} \, du \\ 
& = - \int_0^{2 \pi}  \left( ( yx_{t} )_{u}+\kappa\sqrt{x_{u}^{2}+y_{u}^{2}} \, \right) du \\
& = - y x_t \big\vert_{u=0}^{2\pi} - \int_0^{2 \pi} \kappa \sqrt{x_u^2 + y_u^2} \, du \, = \, 0,
\end{aligned}\nonumber
\end{equation}
where we used (\ref{forparts1}) in the second equality. Therefore, the enclosed signed area $A$ is constant under the flow. Furthermore, since $\gamma_{0}$ is Legendrian, we have $y_0(x_0)_{u} = (z_0)_u$, and it follows that
\begin{equation}
A=-\int_{0}^{2\pi} y_0(x_0)_{u}du=-\int_{0}^{2\pi}(z_0)_{u}du=0. \nonumber
\end{equation}
Thus, each curve $\bar{\gamma}(t)$ encloses zero signed area.

For $t \geq 0$, we define a height function $h(t,u)$ by
\begin{equation}
h(t,u)=z_0(0)+\int_{0}^{u}y(t,w)x_{w}(t,w)dw. \label{defineh}
\end{equation}
By the previous paragraph, each curve $\bar{\gamma}$ encloses zero signed area, so $h(t)$ is a $2\pi$-periodic function in the variable $u$. Using (\ref{forparts1}), we have
\begin{equation}
h_{t}=y(t,u)x_{t}(t,u)-y(t,0)x_{t}(t,0)+\int_{0}^{u}\kappa\sqrt{x_{w}^{2}+y_{w}^{2}}dw. \label{height_derivative}
\end{equation}
We also note that
\begin{align}
h(0,u) & = z_0(u), \label{itslegendrian0} \\
h(t,0) & = h(0,0), \label{itslegendrian1} \\
h_{u}(t,u)  & = y(t,u)x_{u}(t,u). \label{itslegendrian2}
\end{align}

To see the existence part of the theorem, define a family of curves $\gamma:[0,t_{max})\times S^{1}\rightarrow\mathbb{R}^{3}$ by
\begin{equation}
\gamma(t,u)=\left(  \bar{\gamma}(t,u),h(t,u)\right). \label{definesolution}
\end{equation}
It follows from (\ref{itslegendrian0})--(\ref{itslegendrian2}) that
\begin{align*}
\gamma(0,u) & = \gamma_0(u),  \\
z(t,0) & = z(0,0),  \\
z_u(t,u) & = y(t,u) x_u(t,u).
\end{align*}
Hence $\gamma(t,u)$ is a family of Legendrian curves with initial datum $\gamma_0$ and normalization $z(t,0) = z(0,0)$.

Using (\ref{height_derivative}) and the assumption that $\bar{\gamma}$ is a solution to the curve shortening flow, we have
\begin{align}
\gamma_t & = \left(  \,\kappa\frac{-y_{u}}{\sqrt{x_{w}^{2}+y_{w}^{2}}},\,\kappa\frac{x_{u}}{\sqrt{x_{w}^{2}+y_{w}^{2}}},\,y(t,u)x_{t}(t,u)\,\right) \label{gamma_derivative} \\
&  \qquad + \left(  \,0,0,\,-y(t,0)x_{t}(t,0)+\int_{0}^{u}\kappa\sqrt{x_{w}^{2}+y_{w}^{2}}dw\right) . \nonumber
\end{align}
Introducing a family of Legendrian angle functions
\begin{equation}
\lambda(t,u)=-y(t,0)x_{t}(t,0)+\int_{0}^{u}\kappa\sqrt{x_{w}^{2}+y_{w}^{2}}dw \label{definelam}
\end{equation}
and using (\ref{gamma_derivative}) and (\ref{projflow}), we compute $$\gamma_t = \kappa N + \lambda\xi|_{\gamma}.$$ Thus, the family of curves $\gamma(t,u)$ defined in  (\ref{definesolution}) is a solution to the normalized Legendrian curve shortening flow starting at $\gamma_{0}$ with the properties stated in the theorem.

To see that the solution is unique, notice that $x(t,u)$ and $y(t,u)$ are unique since the projection $\bar{\gamma}(t,u)$ is a solution to the curve shortening flow. The uniqueness of $z(t,u)$ follows from the equation $z_{u}=yx_{u}$ and the normalization $z(t,0)=z(0,0)$.
\end{proof}

The proof of this theorem shows there is a correspondence between solutions to the Legendrian curve shortening flow in $\mathbb{R}^{3}$ with $\int_{\gamma} \kappa= 0$ and solutions to the curve shortening flow in $\mathbb{R}^{2}$ with zero winding number that enclose zero signed area.

\bigskip

\section{Curve shortening flow in $\mathbb{R}^{2}$ for balanced figure-eights}

Let $\gamma:[0,t_{max})\times S^{1}\rightarrow\mathbb{R}^{2}$ be a solution to the curve shortening flow:
\[
\gamma_{t}=\overrightarrow{\kappa}.
\]
We assume that the initial curve $\gamma_{0}(u)=\gamma(0,u)$ has the following properties:

\begin{enumerate}
\item[(I.)] The curve $\gamma_{0}$ has winding number zero: $\int_{\gamma_{0}} \kappa_{0} = 0$.

\item[(II.)] The curve $\gamma_{0}$ encloses zero signed area: $-\int_{0}^{2\pi} y x_{u} du =0$.

\item[(III.)] The curve $\gamma_{0}$ has exactly one self-intersection.
\end{enumerate}

\begin{definition}
A closed, plane curve is called a \textbf{balanced figure-eight} it satisfies properties (I.)-(III.) above. 
\end{definition}

The proof of Theorem~\ref{lcsf:exists} shows that the first two of these properties are preserved together under the curve shortening flow. Angenent \cite{A3} \cite{A1} showed that the number of self-intersections is
non-increasing under the curve shortening flow. For a closed plane curve to satisfy the first property, it must have at least one self-intersection. Thus, all three properties are preserved together under the curve shortening flow, and a balanced figure-eight remains a balanced figure-eight until its curvature becomes unbounded at time $t=t_{max}$.

Let $|A|$ denote the total area enclosed by the curve.

\begin{lemma}
\cite[Lemma 2]{G2} For a balanced figure-eight under the curve shortening flow,
\[
-4\pi\leq\frac{d|A|}{dt}\leq-2\pi.
\]
\end{lemma}

\begin{proof}
The lemma follows from the equation $$\frac{d}{dt}|A|= -2 \pi \, - \textrm{twice the interior angle at the point of self-intersection}.$$
\end{proof}

\begin{lemma}
For a balanced figure-eight under the curve shortening flow,
\[
\lim_{t \to t_{max} } |A| =0.
\]
\end{lemma}

\begin{proof}
This statement has been observed in \cite[Lemma 3]{G1} and \cite[pg 205]{A2}. We give a quick heuristic. Since $|A|$ is decreasing and the total signed area is zero, we know that either both loops collapse to zero area at
$t_{max}$, or there they both have a positive lower bound on area as the curve becomes singular. By Theorem \ref{lcsf:exists} the limiting curve consists of crossing points, analytic non-crossing points, and singular points. If
the area is not collapsing, there will be plenty of analytic non-crossing points in the limit. It follows that for a small time before the singular time, one can run curve shortening flow on an embedded arc ending at two
smooth non-crossing points that develops an interior singularity. Blowing up at the singularity, we cannot obtain a self-shrinker, but instead obtain a grim reaper in the limit, following the work of Abresch-Langer \cite{AL} and Altschuler \cite{Alt}. On the reaper, the ratio between the extrinsic and intrinsic distances becomes arbitrary small. Since the curve is smooth at its boundary, this contradicts the maximum principle for the distance ratio on arcs given by Huisken \cite[Theorem 2.1]{Huisken}.
\end{proof}

It is unknown if the curve shortening flow shrinks a balanced figure-eight to a point. Grayson \cite{G2} conjectured that all should.

\bigskip

\subsection{Figure-eights with two symmetries}

We first observe that a figure-eight shape that has symmetries about both an interior and exterior axis at a crossing point must collapse to a point. This is a consequence of the following observations of Grayson \cite[Lemma 5.2]{G0} and Angenent \cite[remark on pg 200]{A2}.

\begin{lemma}
\label{crosspoint}
The limit of a figure-eight curve at singular time is contained in the closure of the set of (presingular time) crossing points.
\end{lemma}

\begin{proof}
Suppose to the contrary that there is some smooth piece of curve with positive length that is not in the limit of the crossing points. Since this must be the limit of two smooth curves, we can write both curves as a graph over a tangent plane at the smooth piece, for time very near the singular times. Both pieces converge smoothly and meet at this point, but this violates the strong maximum principle. 
\end{proof}

The next corollary is immediate.

\begin{corollary}
Any balanced figure-eight whose crossing point remains fixed under the curve shortening flow must shrink to a point.
\end{corollary}

In particular, if a figure-eight is symmetric with respect to reflections about the $x$-axis and the $y$-axis, then it shrinks to a point under the curve shortening flow.

\bigskip

\subsection{Figure-eights with one symmetry}

We impose two conditions on balanced figure-eights:

\begin{enumerate}
\item[(1.)] The curve has two inflection points. That is, the curvature vanishes
only twice.

\item[(2.)] The curve is symmetric across an interior axis. That is, the line that bisects the intersection on the inside of the curve is an axis of symmetry.

\end{enumerate}
We observe that both of these conditions are preserved: That the first condition is preserved is found in \cite{A1}. The second condition is preserved by symmetry along with the observation that curve shortening flow in $\mathbb{R}^2$ is unaffected by reflections and translations.

Before proving the main result, we give a proposition. 

\begin{proposition}
Suppose a closed curve evolves by curve shortening flow on the time interval $\left[  -\tau_{0}/2,0\right]$ and at time $-\tau_{0}/2$ the curve is contained in the rectangle
\[
\mathcal{R}=(-\infty,0]\times\lbrack-C_{0}\tau_{0},C_{0}\tau_{0}].
\]
Then the curve shortening flow moves the curve to the left a distance of
\[
\frac{1}{4C_{0}}+2C_{0}\tau_{0}\log\cos(\frac{1}{2}).
\]
\end{proposition}

\begin{proof}
\ Suppose at some fixed time $-\tau_{0}/2$ the curve is contained in the rectangle
\[
\mathcal{R}=(-\infty,0]\times\lbrack-C_{0}\tau_{0},C_{0}\tau_{0}],
\]
where $C_{0}>0$ is a fixed positive constant. Without loss of generality, we may assume the rightmost points on the curve have $x$-coordinate equal to 0. Consider the grim reaper $\mathcal{G}$ defined by
\[
\mathcal{G}(y,t)=-2C_{0}\tau_{0}\log\cos(\frac{1}{2})+2C_{0}\tau_{0}\log
\cos(\frac{y}{2C_{0}\tau_{0}})-\frac{1}{2C_{0}\tau_{0}}(t+\tau_{0}/2).
\]
Then, the rectangle $\mathcal{R}$ is contained inside the region bounded by the graph of $\mathcal{G}(\cdot,-\tau_{0}/2)$. Applying the maximum principle to the grim reaper and the figure-eight, we conclude that the figure-eight must still be inside the grim reaper at time $t=0$ (recall the flow is defined for $-\tau/2 \leq t \leq 0$). At $t=0$, we have
\[
\mathcal{G}(y,0)\leq\mathcal{G}(0,0)=\mathcal{G}(0,-\tau/2)-\frac{1}{4C_{0}},
\]
so that the grim reaper pushes in (to the left) by $1/(4C_{0})$. If $\tau_{0}$ can be chosen small relative to $C_{0}$ and $1/C_{0}$, then the grim reaper will push past $0$ by the amount
\[
\mathcal{G}(y,0)\leq-2C_{0}\tau_{0}\log\cos(\frac{1}{2})-\frac{1}{4C_{0}}.
\]

\end{proof}

Next, we deal with an intermediate case. 

\begin{proposition}
Let $\theta$ be the angle that the tangent to the curve makes with the $x$-axis. Suppose that, in addition to (1.) and (2.) above, a balanced figure-eight satisfies
\begin{equation}
\text{osc} \, \theta\leq 2 \pi-\varepsilon, \label{oscbound}
\end{equation}
for some $\varepsilon>0$. Then the curve shrinks to a point under curve shortening flow. 
\end{proposition}

\begin{proof}
First, we observe that $\text{osc} \, \theta$ is strictly monotone. This follows from arguments in \cite[Lemma 5]{G2}, applied at both the maximum and minimum of $\theta.$ Notice that (\ref{oscbound}) implies that the crossing angle must be bounded above by $\pi-\varepsilon.$ In fact, (\ref{oscbound}) implies that the curve is contained in conical regions on either side of the crossing point.

We assume here and in the sequel that the $x$-axis is the interior axis across which the curve is symmetric. In particular, the $y$-oscillation of the curve must decay with the $x$-oscillation.

Now consider the curve shortening flow of a balanced figure-eight, satisfying conditions (1.) and (2.) above, that is area collapsing at $t=t_{max}$. \ Let $\tau=t_{max}-t$. Then
\begin{equation}
\pi\tau\leq A_{i}(t)\leq2\pi\tau,\qquad i=1,2,\label{areabound}
\end{equation}
where $A_{1}(t)$ and $A_{2}(t)$ are the areas of the two loops of the figure-eight at time $t$.

At a given time, position the curve so that the crossing is at the origin, the axis of symmetry is the $x$-axis, and the loop with the largest projection onto the $x$-axis opens to the right. Our first goal is to show that one loop contains a large convex piece. Let $\ell$ denote the length of the projection onto the $x$-axis. Note that by the angle condition, the region to the right of the origin must be completely contained in the right-hand plane. We deal with several cases.

\begin{enumerate}
\item[(a.)] Both inflection points lie at or to the left of the origin: In this case, the region to the right of the origin is convex. This region has projection onto the $x$-axis of length at least $\ell/2.$

\item[(b.)] Both inflection points lie to the right of the origin, but with $x$-coordinate less than $\ell/4$: In this case, we still have a large portion of the right region that is convex. In fact, the projection of this convex region onto the $x$-axis has length at least $\ell/4.$

\item[(c.)] Both inflection points lie to the right of the origin, but with $x$-coordinate at least $\ell/4$: For this case, the left region is convex. We claim it must also have projection onto the $x$-axis of length at least $\ell/4$. If not, then we will be able to reflect it across the $y$-axis and it will lie strictly inside of the right region. This violates the second balancing condition (II.).
\end{enumerate}
In any of the above cases, we have a convex subregion that projects onto the $x$-axis with length at least $\ell/4$.

Next, translate and reflect if necessary, so that the tip of the loop in question is at the origin. We claim that the loop in question is contained inside the rectangle
\[
\mathcal{R}=(-\infty,0]\times\lbrack-\frac{8\pi\tau}{\ell},\frac{8\pi\tau}{\ell}].
\]
To see this, suppose the maximum $y$-coordinate of the loop in question is $h$. Then, the area inside of this loop must be at least $h \ell /4$. But the area is bounded above by $2\pi\tau$, so we have
\[
h\leq\frac{8\pi\tau}{\ell}.
\]

Now, run curve shortening flow forward in time $\tau/2$, comparing to the grim reaper in the previous proposition. Observe that the maximum principle prohibits the reaper from making an interior touching on the right side of the figure-eight. The reaper will not touch the crossing point on the $x$-axis, so we need not fret if the left portion of the curve is not contained in the rectangle. The argument above tells us the curve must translate to the left by
\begin{equation}
\frac{\ell}{32\pi}+\frac{16\pi}{\ell}\tau\log\cos(\frac{1}{2}). \label{decrease}
\end{equation}
It follows that the length of the projection of the curve has decreased by (\ref{decrease}) in half of the time to extinction.

Now, repeat this argument, proceeding half of the time to extinction. If $\ell \geq\delta>0$, then as $\tau\rightarrow0$ the quantity (\ref{decrease}) will be bounded below by a positive value, implying a contradiction. Thus, $\ell \rightarrow 0$. 

Finally, since the $y$-oscillation decays with the $x$-oscillation, we conclude that the curve shrinks to a point. 
\end{proof}

\begin{corollary}
Suppose at any point in time, a balanced figure-eight curve satisfying (1.) and (2.) has a double inflection point. Then the curve must shrink to a point.
\end{corollary}

\begin{proof}
Because there are no other inflection points, it follows that both sides are convex. Now, the crossing angle can be no more than $2\pi$, so  by \cite[Lemma 5]{G2}, the oscillation will be strictly less than $2\pi$ after a short time. Convergence follows from the previous proposition.
\end{proof}

\bigskip

Now, we are ready to prove the main theorem.

\begin{proof}
[Proof of Theorem 1]If there is a double inflection point, we are done. So, we assume that the curvature at the crossing point has a sign. It follows that one region, (without loss of generality, the left one) must be convex.
Now, let the curve shortening flow run without repositioning the crossing point. As long as the curvature remains nonzero at the crossing point, the crossing point will move to the left. Since the crossing points are moving monotonically to the left inside a bounded region, they converge to a unique limit point, and it follows from Lemma~\ref{crosspoint} that the entire curve shrinks to a point. 
\end{proof}

We can refine the reaper argument slightly to get a rate of convergence.

\begin{theorem}
\label{collapse_rate}
Suppose that, in addition to (1.) and (2.) above, a balanced figure-eight satisfies the oscillation bound (\ref{oscbound}). Then there is an $\alpha_{0} > 0$ such that
\begin{equation}
\limsup_{\tau\rightarrow0}\frac{\ell(\tau)}{\tau^{\alpha}}\leq1, \label{the}
\end{equation}
for all $\alpha<\alpha_{0}.$
\end{theorem}

\begin{proof}
Recall that $\ell$ is the length of the the projection on the $x$-axis and $\tau$ is the time remaining until the area collapses. It follows from the maximum principle that $\ell(\tau)$ is monotone. We assume that the flow is defined on $\left[  T_{0},0\right]$, for some negative number $T_0$.

\bigskip

Step 1. \ Oscillation decay based on the grim reaper argument. \ At any $\tau$ we may run the flow for time $\tau/2$ and the length will have decreased by $\frac{\ell}{32\pi}+\frac{16\pi}{\ell}\tau\log\cos(\frac{1}{2})$. We write
\begin{align*}
\frac{\ell}{32\pi}+\frac{16\pi}{\ell}\tau\log\cos(\frac{1}{2}) &  = \ell \left(  \frac{1}{32\pi}+16\pi\log\cos(\frac{1}{2})\frac{\tau}{\ell^{2}}\right)  \\
&  = \ell \left(  c_{1}-c_{2}\frac{\tau}{\ell^{2}}\right),
\end{align*}
so that
\begin{equation}
\ell(\tau/2)\leq\left[  1 - c_1 + c_2 \frac{\tau}{\ell(\tau)^2} \right]  \ell(\tau), \label{decay}
\end{equation}
where $c_1 = \frac{1}{32\pi}$ and $c_2 = 16\pi \log\cos(\frac{1}{2})$.

\bigskip

Step 2. \ For small $\varepsilon > 0$, there exists $\tau_{0}>0$ such that $\ell(\tau_{0})\leq 1$. \ Without loss of generality, we may assume that $\ell(T_0) > 1$. We also assume $\varepsilon \in (0,-T_0)$ is chosen small enough that
\[
\eta:=1 - c_{1} + c_{2} \varepsilon < 1.
\]
Then, we let
\begin{align*}
\tau_{0}  &  =\frac{\varepsilon}{2^{k_{0}}} \\
k_{0}  &  =\lceil\frac{-\log \ell(T_{0})}{\log\eta}\rceil.
\end{align*}
We can check that 
\[
\ell(\tau_{0}) \leq 1.
\]
Assume not, then $\ell(\tau)\geq1$ for all $\tau\in\lbrack\tau_{0},\varepsilon]$ and
\[
\varepsilon\geq\frac{\varepsilon}{\ell^{2}(\tau)}\geq\frac{\tau}{\ell^{2}(\tau)}.
\]
So, we may repeatedly apply the decay estimate from Step 1 to conclude that
\[
\ell \left(  \frac{\varepsilon}{2^{k_0}}\right)  \leq\eta^{k_0} \ell(\varepsilon),
\]
Thus,
\[
\ell(\tau_{0})\leq\eta^{\lceil\frac{-\log \ell(T_{0})}{\log\eta}\rceil} \ell(\varepsilon)\leq\frac{\ell(\varepsilon)}{\ell(T_{0})}\leq1,
\]
which contradicts our assumption that this estimate failed.

\bigskip

Step 3. \ Next we use induction to prove
\[
\ell\left(  \frac{\tau_{0}}{2^{k}}\right)  \leq\eta^{k}.
\]
The first step, $k=0$ is Step 2. \ Now suppose that this fails for the first time at $k,$ that is
\[
\ell\left(  \frac{\tau_{0}}{2^{k-1}}\right)  \leq\eta^{k-1}
\]
while
\begin{equation}
\ell\left(  \frac{\tau_{0}}{2^{k}}\right)  >\eta^{k}. \label{tocontra}
\end{equation}
Necessarily,
\[
\ell\left(  \frac{\tau_{0}}{2^{k}}\right)  >\eta \ell\left(  \frac{\tau_{0}}{2^{k-1}}\right),
\]
which, using (\ref{decay}), implies that
\[
\frac{\tau_{0}/2^{k-1}}{\ell^{2}(\frac{\tau_{0}}{2^{k-1}})} > \varepsilon.
\]
Then
\[
\ell \left( \frac{\tau_{0}}{2^{k-1}} \right) \, < \, \sqrt{ \frac{\tau_{0}}{\varepsilon 2^{k-1}} }  \, = \,  \left( \frac{1}{\sqrt{2}}\right)^{k_0 + k-1} \, \leq \, \left( \frac{1}{\sqrt{2}}\right)^{k},
\]
where we used $\tau_0 = \varepsilon / 2^{k_0}$ and $k_0 \geq 1$. However,
\[
\ell \left( \frac{\tau_{0}}{2^{k}} \right)  \, \leq \, \ell \left( \frac{\tau_{0}}{2^{k-1}} \right) \, <  \, \left( \frac{1}{\sqrt{2}}\right)^{k},
\]
which contradicts (\ref{tocontra}) since $\eta = 1 - c_1 + \varepsilon c_2 > 1 -\frac{1}{32 \pi} > 1 / \sqrt{2}$.

\bigskip

Step 4. \ If
\[
\alpha<\alpha_{0}=\frac{-\log(1-c_{1})}{\log2},
\]
then
\[
\limsup_{ \tau \to 0} \frac{\ell(\tau)}{\tau^{\alpha}} \leq 1.
\]

\bigskip

To see this, first choose $\delta > 0$ so that $\alpha + \delta < \alpha_0$. Then, choose $\varepsilon>0$ small enough that $$\alpha + \delta < \frac{-\log \eta}{\log 2}.$$ 

Now, for small $\tau>0$, choose $k$ such that
\[
\tau\in(\frac{1}{2^{k+1}}\tau_{0},\frac{1}{2^{k}}\tau_{0}].
\]
Then
\[
\ell(\tau) \leq \eta^{k} < \left( \frac{1}{2^k} \right)^{\alpha + \delta} \leq \left( \frac{2 \tau}{\tau_0} \right)^{\alpha + \delta} = \left( \frac{2^{\alpha + \delta} }{\tau_0^{\alpha + \delta}} \tau^\delta \right) \tau^\alpha.
\]
In particular, when $\tau$ is sufficiently small, we have
\[
\frac{2^{\alpha + \delta} }{\tau_0^{\alpha + \delta}} \tau^\delta \leq 1,
\]
and it follows that
\[
\ell(\tau)\leq\tau^{\alpha}.
\]
Therefore,
\[
\limsup_{ \tau \to 0} \frac{\ell(\tau)}{\tau^{\alpha}} \leq 1.
\]
Finally, we note that
\[
\alpha_{0}=\frac{-\log(1-c_{1})}{\log2}=\frac{-\log(1 - \frac{1}{32\pi})}{\log2} \approx0.0 \allowbreak1 \allowbreak 442\,3.
\]
\end{proof}

\bigskip

\section{Isoperimetric estimate for general figure-eights evolving by the curve shortening flow}

In this section we carry out Grayson's argument in \cite{G2} to obtain a quantitative estimate on the rate of blow-up of the isoperimetric profile. We consider the curve shortening flow starting from a balanced figure-eight that is area collapsing at $t=t_{max}.$ \ Let $\tau=t_{max}-t$, and let
\[
Q(\tau)=\frac{L^{2}}{|A|}
\]
be the isoperimetric constant. In \cite{G2}, Grayson showed that the isoperimetric profile $Q(\tau) \to \infty$ as $\tau \to 0$. The following is an attempt to quantify this result.

\begin{theorem}
Suppose a balanced figure-eight satisfies (\ref{oscbound}) and the curve shortening flow exists for  $T_0 \leq t < 0$. Let $\tau = -t$, and set $\tau_0 = -T_0 > 0$. Given $M >0$, if $\alpha \geq 0$ satisfies
\[
\alpha < \frac{ \pi  }{ 4 \sqrt{3} \ln2} \, \frac{ e^{-4 \pi M / \tau_0^\alpha} }{ \left( \text{osc} \, \theta(\tau_{0})-\pi\right) },
\]
then
\[
Q(\tau)\geq M \tau^{-\alpha},
\]
for some $\tau\in (0,\tau_0]$.
\end{theorem}

\begin{proof}
Let $\theta$ be the angle that the tangent to the curve makes to the $x$-axis. This is a single-valued function, and we can choose the curve to attain the minimum value of $0$ at the origin. We will consider the connected piece of the curve that is graphical over the $x$-axis in this arrangement. Let
\begin{align*}
a_{t} &  =\text{maximum value of }-x\text{ on the curve at time }t.\\
b_{t} &  =\text{maximum value of }x\text{ on the curve at time }t.
\end{align*}
Necessarily,
\[
a_{t},b_{t}<L(t)/2.
\]
We know that $a_{t},b_{t}$ are non-increasing, so following \cite{G2} we have that $\theta$ satisfies a parabolic equation forward in time defined on non-increasing subintervals of $\left[  -a_{t},b_{t}\right]$. Namely,
\[
\frac{d}{dt}\theta(x,t)=\cos^{2}\theta\frac{d^{2}}{dx^{2}}\theta(x,t).
\]

To estimate the minimum of $\theta$, we construct a comparison function. Let $g_{0}$ be the function
\begin{align*}
g_{0}(x) &  =0\text{ on }\left[  -1,1\right],  \\
g_{0}(x) &  =\frac{\pi}{4}\text{ on }\left\vert x\right\vert \geq 1,
\end{align*}
and define
\[
g(x,t)=\frac{1}{\sqrt{4\pi t}}\int_{\mathbb{R}} e^{-\left\vert x-y\right\vert ^{2}/(4t)}g_{0}(y)dy.
\]
We notice that $g$ is a solution to heat flow
\[
\frac{d}{dt}g(x,t)=\frac{d^{2}}{dx^{2}}g(x,t),
\]
and we use this to build a scaled comparison function
\[
f_{\beta}(x,t)=g(\sqrt{\beta}x,\frac{1}{2}\beta t),
\]
which satisfies
\[
\frac{d}{dt}f(x,t)=\frac{1}{2}\frac{d^{2}}{dx^{2}}f(x,t).
\]
Further, we have
\begin{align*}
f(x,0) &  =0\text{ on }\left[  -\frac{1}{\sqrt{\beta}},\frac{1}{\sqrt{\beta}}\right] , \\
f(x,0) &  =\frac{\pi}{4}\text{ on }\left\vert x\right\vert \geq\frac{1}{\sqrt{\beta}}.
\end{align*}
We choose
\[
\beta=\frac{4}{L^{2}},
\]
so that $\left[ -a_{t},b_{t}\right]  \subset\left[  -\frac{1}{\sqrt{\beta}},\frac{1}{\sqrt{\beta}}\right]$. Since $\theta$ approaches $\pi/2$ on the endpoints of $\left[-a_{t},b_{t}\right]$, we can apply a maximum principle argument (c.f. \cite[Lemma 5]{G2}) to see that
\[
\theta\geq f_{\beta}.
\]
In particular, at $\tau/2$ we have
\[
\theta(x,\tau/2)\geq f_{\beta}(x,\tau/2),
\]
so that
\[
\min\theta(\tau/2)\geq\min f_{\beta}(x,\tau/2)=f_{\beta}(0,\tau/2).
\]
Now,
\begin{align*}
f_{\beta}(0,\tau/2) &  =g(0,\frac{\tau}{L^{2}})\\
&  =\frac{1}{\sqrt{4\pi\frac{\tau}{L^{2}}}}\int e^{-\left\vert y\right\vert^{2}/( 4 \frac{\tau}{L^{2}})}g_0(y)dy\\
&  =\frac{\pi}{4}\frac{1}{\sqrt{4\pi\frac{\tau}{L^{2}}}}\int_{\left\vert y\right\vert \geq1}e^{-\left\vert y\right\vert ^{2}/(4 \frac{\tau}{L^{2}} )  }dy \\
& =\frac{\sqrt{\pi}}{8} \frac{L}{\sqrt{\tau}}   \int_{\left\vert y\right\vert \geq 1}e^{-\left\vert y\right\vert ^{2}L^{2}/(4\tau)}dy,
\end{align*}
and we estimate
\begin{align*}
\frac{\sqrt{\pi}}{8} \frac{L}{\sqrt{\tau}} \int_{\left\vert y\right\vert \geq 1}e^{-\left\vert y\right\vert ^{2}L^{2}/(4\tau)}dy &  \geq \frac{\sqrt{\pi}}{8} \frac{L}{\sqrt{\tau}} 2e^{-L^{2}/\tau} \\
&  = \frac{\sqrt{\pi}}{4}\frac{L}{\sqrt{\tau}}e^{-L^{2}/\tau},
\end{align*}
using the simple bound
\[
\int_{1}^{\infty}e^{-\delta z^{2}}dz\geq\int_{1}^{2}e^{-\delta z^{2}}dz \geq e^{-4\delta}.
\]
Thus,
\[
\min\theta(\tau/2)\geq\frac{\sqrt{\pi}}{4}\frac{L}{\sqrt{\tau}}e^{-L^{2}/\tau}.
\]

Writing $\tau_{k}=\tau_{0} / 2^k$, $L_{j}=L(\tau_{j})$ and using the global maximum principle for the single-valued function $\theta$, we deduce
\[
\text{osc} \, \theta(\tau_{j+1})\leq\text{osc} \, \theta(\tau_{j})-\frac{\sqrt{\pi}}{4}\frac{L_{j}}{\sqrt{\tau_{j}}}e^{-L_{j}^{2}/\tau_{j}},
\]
which can be iterated to conclude that
\[
\text{osc} \, \theta(\tau_{k})\leq\text{osc} \, \theta(\tau_{0})-\frac{\sqrt{\pi}}%
{4}\sum_{j=0}^{k-1}\frac{L_{j}}{\sqrt{\tau_{j}}}e^{-L_{j}^{2}/\tau_{j}}.
\]
Therefore,
\[
\frac{\sqrt{\pi}}{4}\sum_{j=0}^{k-1}\frac{2^{j/2}L_{j}}{\sqrt{\tau_{0}}%
}e^{-2^{j}L_{j}^{2}/\tau_{0}}\leq\text{osc} \, \theta(\tau_{0})-\text{osc} \, \theta(\tau_{k}).
\]
In terms of
\[
q_{j}=\frac{2^{j/2} L_j}  {\sqrt{\tau_{0}}},
\]
the above becomes
\begin{equation}
\sum_{j=0}^{\infty}q_{j}\exp(-q_{j}^{2})\leq\frac{4}{\sqrt{\pi}}\left(\text{osc} \, \theta(\tau_{0})-\pi\right).\label{five}
\end{equation}

Now, we begin to build a contradiction. Given $M>0$, assume that for all $\tau\leq\tau_{0}$ we have
\[
\tau^{\alpha}Q(\tau)\leq M.
\]
Then for the subsequence $\tau_j = 2^{-j}\tau_{0}$, we have
\begin{equation}
\left(  \tau_j \right)  ^{\alpha}Q(\tau_j) \leq M. \label{to contradict}
\end{equation}
We know from (\ref{areabound}) that $A(\tau_{j})=c_{j} \tau_0 2^{-j}$ for some constant $c_{j}\in\lbrack\pi,2\pi]$.  Thus,
\[
Q_{j}:=Q(2^{-j}\tau_{0})=\frac{L_{j}^{2}}{c_{j} \tau_0 2^{-j}}
=\frac{q_{j}^{2}}{c_{j}}.
\]
Using (\ref{to contradict}), we have
\[
\frac{q_{j}^{2}}{c_{j}}=Q_{j}\leq M\left(  \frac{2^{j}}{\tau_{0}}\right)^{\alpha}=\frac{M \mu^j}{\tau_{0}^{\alpha}}
\]
and 
\begin{equation}
q_{j}^{2} \leq \frac{c_j M \mu^j}{\tau_{0}^{\alpha}}  \leq \frac{2 \pi M \mu^j}{\tau_{0}^{\alpha}}, \label{six}
\end{equation}
where $\mu = 2^\alpha$.

Next, consider the function
\[
s \mapsto s\exp(-s^{2}),
\]
which is decreasing for $s>1/\sqrt{2}$. The isoperimetric inequality guarantees $$q_j = \sqrt{c_j Q_j} \geq 2 \pi > \frac{1}{\sqrt{2}},$$ so (\ref{six}) implies
\[
\sum_{j=0}^{\infty}q_{j}\exp(-q_{j}^{2})\geq\sum_{j=0}^{\infty}\sqrt{\frac{2 \pi M \mu^j}{\tau_{0}^{\alpha}}} \exp(-\frac{2 \pi M \mu^j}{\tau_{0}^{\alpha}}).
\]
Thus, (\ref{to contradict}) together with (\ref{five}) leads to
\begin{equation}
\frac{4}{\sqrt{\pi}}\left(  \text{osc} \, \theta(\tau_{0})-\pi\right) \geq \sqrt{\frac{2 \pi M}{\tau_{0}^{\alpha}}} \, \sum_{j=0}^{\infty}\mu^{j/2}\exp\left(-\frac{2 \pi M\mu^j}{\tau_{0}^{\alpha}}\right).\label{eight}
\end{equation}

Notice that the assumption $\tau^{\alpha}Q(\tau)\leq M$, for $\tau \leq \tau_0$, and the isoperimetric inequality imply
\[
\frac{2 \pi M \mu^s }{\tau_{0}^{\alpha}} \geq 2 \pi Q(\tau_0) \mu^s \geq 8 \pi^2 \mu^s \geq 8 \pi^2,
\]
when $s>0$. In addition, observe that the function
\[
s \mapsto \mu^{s/2}\exp\left(  -\frac{2 \pi M \mu^s}{\tau_{0}^{\alpha}} \right) = \exp\left( \frac{s}{2} \ln \mu -\frac{2 \pi M \mu^s}{\tau_{0}^{\alpha}} \right)
\]
is decreasing when $\frac{2 \pi M \mu^s}{\tau_{0}^{\alpha}} > 1/2$. Consequently,
\begin{align}
\sum_{j=0}^{\infty}\mu^{j/2}\exp\left(  -\frac{2 \pi M \mu^j}{\tau_{0}^{\alpha}} \right)   &  \geq \int_{0}^{\infty} \mu^{s/2} \exp\left(  -\frac{2 \pi M \mu^s}{\tau_{0}^{\alpha}} \right)  ds\label{seven} \\
&  =\int_{1}^{\infty}\exp\left(  -\frac{2 \pi M}{\tau_{0}^{\alpha}} u^{2}\right)  \frac{du}{\frac{1}{2}\ln\mu}. \nonumber 
\end{align}
Using the estimate
\begin{align*}
\left(  \int_{1}^{\infty}e^{-\delta z^{2}}dz\right)  ^{2} &  =\int_{1}^{\infty} \int_{1}^{\infty} e^{-\delta(z^{2}+w^{2})}dzdw \\
&  \geq\int_{\pi/6}^{\pi/3}\int_{2}^{\infty}e^{-\delta r^{2}}rdrd\varphi \\
&  =\frac{\pi}{6}\frac{e^{-4\delta}}{2\delta} = \frac{\pi e^{-4\delta}}{12\delta},
\end{align*}
we have from (\ref{seven}) that
\begin{align}
\label{seven2}
\sum_{j=0}^{\infty}\mu^{j/2}\exp\left(  -\frac{2 \pi M \mu^j}{\tau_{0}^{\alpha}} \right)
&  \geq\frac{2}{\ln\mu}\left(  \frac{\pi e^{-4 (\frac{2 \pi M}{\tau_{0}^{\alpha}})}}{12\left(  \frac{2\pi M}{\tau_{0}^{\alpha}} \right)}\right)^{1/2} \\
&  =\frac{1}{\ln\mu} \sqrt{ \frac{\tau_{0}^{\alpha} }{6M} } \, e^{-4 \pi M / \tau_{0}^{\alpha} }. \nonumber
\end{align}

Finally, combining (\ref{seven2}) with (\ref{eight}), we have
\[
\frac{4}{\sqrt{\pi}}\left(  \text{osc} \, \theta(\tau_{0})-\pi\right)  \geq \sqrt{\frac{2 \pi M}{\tau_{0}^{\alpha}}} \left( \frac{1}{\ln\mu}\sqrt{\frac{\tau^{\alpha} }{6M}} \, e^{-4 \pi M / \tau_{0}^{\alpha} } \right),
\]
so that
\[
\alpha \ln 2 = \ln\mu \geq  \frac{ \pi  }{ 4 \sqrt{3} } \, \frac{ e^{-4 \pi M/\tau_{0}^{\alpha}} }{ \left( \text{osc} \, \theta(\tau_{0})-\pi\right) }.
\]
We arrive at a contradiction if $\alpha$ satisfies
\[
\alpha < \frac{ \pi  }{ 4 \sqrt{3} \ln 2} \, \frac{ e^{-4 \pi M / \tau_0^\alpha} }{ \left( \text{osc} \, \theta(\tau_{0})-\pi\right) }.
\]

\end{proof}

\bigskip

\section{Gradient flows associated with Legendrian curve shortening flow}
\label{GF}

In this section, we discuss three different gradient flows associated with Legendrian curve shortening flow. Given a Legendrian immersion $\gamma$ as above, we can consider a normal variation of Legendrian immersions
\[
\Gamma(u,t):S^{1}\times(-\varepsilon,\varepsilon)\rightarrow{\mathbb{R}}^{3},
\]
\[
\Gamma(u,0) = \gamma(u).
\] 
The normal variation field can be written as
\[
\frac{d\Gamma}{dt}(u,0)= \phi(u)N(u)+f(u) \xi(u),
\]
and the calculation in Lemma~\ref{Lemma1} shows that this is a variarion of Legendrian immersions if and only if 
\[
\phi = \frac{1}{|\partial u|_g} f_{u}.
\]

If we consider the manifold of $C^{1}$ Legendrian immersions near a given $\gamma$, it follows that the tangent space can be parameterized by the set of $C^{1}$ functions $f$ to which we associate the normal variation fields
\[
V=\frac{1}{|\partial u|_g} f_{u}(u)N(u)+f(u) \xi(u).
\]
Note that the length functional
\[
L=\int \sqrt{x_{u}^{2}+y_{u}^{2}}du
\]
has differential
\[
dL(f)=-\int\kappa f_{u} du = - \int_{S^{1}} \kappa f_{s},
\]
where $s$ is the arclength parameter. Now, there are three distinct ways to put a metric on this space and produce a gradient flow. One possibility, which is explored in \cite{Le}, is to use the $\left(  L^{2},\sqrt{x_{u}^{2}+y_{u}^{2}}du \right)$ metric on the functions $f$, i.e.
\[
\langle f ,  h \rangle=\int f(u)h(u)\sqrt{x_{u}^{2}+y_{u}^{2}}du.
\]
The negative gradient flow of the length functional with respect to this metric, leads to a fourth order equation as follows, see \cite{Le}. The gradient of $L$ is a function $\zeta\in L^{2}$ such that
\[
\langle\zeta,f\rangle=dL(f).
\]
That is,
\[
\int_{S^{1}} \zeta f = -\int_{S^{1}} \kappa f_{s}.
\]
Integrating by parts, we see that
\[
\int_{S^{1}} \zeta f= \int_{S^{1}} \kappa_{s}f
\]
or
\[
\zeta=\kappa_{s}.
\]
Now, in the Legendrian normal direction, the curve is evolving with speed
\[
-\zeta_{s}=-\kappa_{ss},
\]
so the evolution equation for $\kappa$ can be computed using the standard formula in \cite[Cor 3.5]{Huisken84} \cite{GuanLi}
\begin{align*}
\frac{d}{dt}\kappa &  =\Delta_{g}(-\zeta_{s}) +\kappa^{2}(-\zeta_{s}) \\
&  = - \Delta_{g}^{2}\kappa - \kappa^{2}\Delta_{g}\kappa.
\end{align*}

A second possibility is to consider the full $L^{2}$ metric measuring the deformation, i.e
\[
\langle f,h\rangle=\int_{S^{1}}\left(  f h +f_{s} h_{s}\right).
\]
This defines a different sort of gradient flow equation. We require that
\[
\int_{S^{1}}f\zeta+f_{s}\zeta_{s}=-\int_{S^{1}} \kappa f_{s},
\]
which leads to the system
\begin{align*}
\zeta-\zeta_{ss}  & =\kappa_{s}, \\
\frac{d}{dt}\kappa & = -\Delta_{g}\zeta_{s} -\kappa^{2}\zeta_{s}.
\end{align*}

Finally, if one considers the $L^{2}$ metric, which only measures the deformation in the Legendrian normal direction, then
\[
\langle f,h\rangle=\int_{S^{1}} f_{s} h_{s}.
\]
Notice that this metric is indefinite in the sense that the constant functions are null vectors. Computing the negative gradient flow for the length functional with respect to this metric, we arrive at the relation $$\zeta_s = - \kappa$$ and the evolution equation for curve shortening flow $$\frac{d}{dt}\kappa = \Delta_{g} \kappa +\kappa^{3}.$$ We remark that while the metric may seem contrived to produce mean curvature flow on the base space, this metric is often studied and is very natural in the setting of Sasaki manifolds.

\bigskip

\section{Legendrian mean curvature flow}
\label{lemcf}

In this section we outline a proof of the short time existence for Legendrian mean curvature flow in Sasaki-Einstein manifolds in the case where the initial Legendrian manifold has zero Maslov class. Legendrian mean curvature flow in Sasaki manifolds was introduced and studied by Smoczyk in \cite{Smoczyk}.  

Let $(S, g)$ be a Sasaki manifold of dimension $(2n+1)$. By definition, its metric cone $(X_0\cong \R_+\times S, dr^2+r^2 g)$ is a K\"ahler manifold with  a compatible complex structure $J$. We use the notation $X=X_0\cup\{r=0\}$ to denote the cone with the vertex, and we let $g_X=dr^2+r^2 g$ denote the metric and $\omega_X$ denote the K\"ahler form. We also identify $S$ with the hypersurface $\{r=1\}$ in $X_0$. 

In terms of the homothetic vector field $r\p_r$, the Reeb vector field $\xi$ is defined to be 
\[
\xi=J\left(r\p_r\right).
\]
Note that $r\p_r$ and $\xi$ are both real holomorphic, and $\xi$ is Killing. The dual 1-form of $\xi$ is denoted by $\eta$, so that
\[
\eta=r^{-2}g_X(\xi, \cdot)=J\left(\frac{dr}{r}\right). 
\]
When restricted to $S$, the 1-form $\eta$ defines a contact structure. When a Sasaki structure $(S, g)$ is restricted to the contact subbundle $\text{Ker}(\eta)$, it inherits a transverse K\"ahler structure via
\[
g=\eta\otimes \eta+g^T.
\] The transverse K\"ahler form is given by
\[
\omega^T=\frac{1}{2}d\eta.
\]

The metric cone $(X, g_X, J)$ is called a Calabi-Yau cone if there exists a holomorphic (n+1,0)-form $\Omega$ on $X_0$ such that 
\[
\omega^{n+1}_X =c_n \Omega\wedge \bar \Omega,
\]
where $c_n = (i/2)^{n+1}(-1)^{n(n+1)/2}$. Notice that this condition implies that $g_X$ has zero Ricci curvature and $(S, g)$ is a Sasaki-Einstein manifold with $Ric= 2ng$. Also, since $\omega_X$ is homogeneous degree $2$ under the vector field $r\p_r$, the holomorphic form $\Omega$ is homogeneous degree $(n+1)$. 

From now on, we fix a Calabi-Yau cone $(X, g_X, J)$ and the corresponding Sasaki-Einstein manifold $(S, \xi, \eta, g)$. Given a submanifold $L$ of $S$, we consider a subcone $Y$ of $X$ such that $Y_0\cong \R_{+}\times L$ via the identification of $S$ with the hypersurface $\{r=1\}$ in $X_0$. By definition $L$ is a Legendrian submanifold of $S$ if $\eta=0$ on $L$. We have the following correspondence between Legendrian submanifolds in $S$ and Lagrangian subcones in $X$:

\begin{center}
$L\subset S$ is Legendrian if and only if $Y_0\subset X_0$ is Lagrangian.
\end{center}

Now, suppose $Y_0$ is a Lagrangian subcone of $X_0$ and $L$ is the corresponding Legendrian submanifold of $S$. The Calabi-Yau condition implies that 
\[
\Omega |_{Y_0} =e^{i \theta} dvol,
\]
where $\theta: Y_0\rightarrow S^1\cong\R/\Z$ is an $S^1$ valued function, called the Lagrangian angle. Since $\Omega|_{Y_0}$ and $dvol$ are both homogeneous degree $(n+1)$, the Lagrangian angle $\theta$ is independent of $r$, and we can then view it as an $S^1$ valued function on $L$, which we call the Legendrian angle $\theta: L\rightarrow S^1$.  In particular, $[d\theta]$ defines a cohomology class in $H^1(Y_0, \R)\cong H^1(L, \R)$, the Maslov class of $Y_0/L$. For a Lagrangian, the relation between the mean curvature vector of $Y_0$ and the Lagrangian angle $\theta$ is given by $H_{Y_0}=J\overline{\nabla} \theta$. 

\begin{proposition}
When restricted to $\{r=1\}$, the mean curvature vectors of $Y_0$ and $L$ are related by \[H_{Y_0}=H_L.\]
\end{proposition}

Recalling the fact that $H_L$ is perpendicular to $\xi$, we have $H_{Y_0}$ is perpendicular to the complex line generated by $\{r\p_r, \xi\}$, and consequently $d\theta(\xi)=0$. 

\begin{definition}
When the Maslov class $[d\theta]=0$, we say that a family of Legendrian submanifolds $F_t: L\rightarrow S$ is a solution to \textbf{Legendrian mean curvature flow} if
\begin{equation}\label{E-lmcf}
\left( \frac{\p F_t}{\p t} \right)^{\perp}=H+2 \theta \xi.
\end{equation}
\end{definition}

The short time existence and uniqueness of Legendrian mean curvature flow was proved in \cite{Smoczyk} by quoting standard PDE theory.  We will show the above equation is well-posed and has a unique smooth short time solution by adapting results from Behrndt's thesis \cite{Behrndt} to the setting of Legendrian mean curvature flow.

First, we need a Legendrian neighborhood theorem.

\begin{theorem}\label{T-lnb} Let $(M, \eta)$ be a $(2n+1)$ dimensional contact manifold, and let
$L$ be a Legendrian submanifold, i.e. $\text{dim}\;L =n$ and $\eta|_{TL} \equiv 0$. Then there is a
neighborhood $U$ of $L$ in $M$ and a diffeomorphism
\[
\Psi: \R\times T^*L\rightarrow U
\]
so that $\Psi^*\eta=\eta_0:=ds+\beta_0$, where $s$ is the coordinate in $\R$ and $\beta_0$ is the canonical
1-form on $T^*L$.\end{theorem}

This theorem is stated for an embedded Legendrian submanifold, but there is also a version for immersions (the only difference is that we only require $\Psi$ to be an immersion instead of a diffeomorphism). Given the Legendrian neighborhood theorem, we define a map $I: L\rightarrow \R\times T^*L$ by $I(x) = (f(x), x, \beta(x))$. 

\begin{proposition}
 $(\Psi \circ I) (L)$ is Legendrian if and only if  $df+\beta=0$. 
\end{proposition}

\begin{proof}
Using the Legendrian neighborhood in Theorem \ref{T-lnb}, the restriction of $\eta$ to $T(\Psi \circ I) (L)$ is  given by $$\eta |_{T(\Psi \circ I) (L)} = df+\beta.$$
\end{proof}

Given a family of functions $f_t$, if we define the embeddings $I_{t}: L\rightarrow \R\times T^*L$ by $I_t(x)=(f_t(x), x, -df_t(x))$, then it follows from the previous proposition that $F_t=\Psi \circ I_{t}$ is a family of Legendrian submanifolds of $S$. Conversely, if $F_t$ is a family of Legendrian submanifolds, then modulo diffeomorphisms of $L$, for a short time, the family $F_t$ is induced by a family of functions $f_t$ via the embeddings $I_t: L\rightarrow \R\times T^{*}L$. The well-posedness of \eqref{E-lmcf} is due to the following fact as in \cite[Lemma 3.2 and Corollary 3.3]{Smoczyk}.

\begin{proposition}[Smoczyk \cite{Smoczyk}]\label{P-smoczyk}
Suppose $F_t: L\rightarrow S$ is a family of Legendrian submanifolds induced by a family of functions $f_t$. If the Maslov class of $[d\theta_t]$ is zero at $t=0$, then $F_t$ has zero Maslov class, for any $t>0$, and the Legendrian angle $\theta_t$ is a well-defined $\R$-valued function. 
\end{proposition}

With this preparation, we can define a curvature flow induced by the functions $f_t$, which is equivalent to \eqref{E-lmcf}. 

\begin{proposition}
Suppose $f_t$ is a family of functions satisfying 
\begin{equation}\label{E-zero}
\frac{\p f_t}{\p t}= 2 \theta \circ \Psi(I_{f_t}),
\end{equation}
then $F_t=\Psi\circ I_{f_t}$ satisfies the curvature flow \eqref{E-lmcf}.
\end{proposition}

\begin{proof}
We note that by Proposition \ref{P-smoczyk}, if the initial data $F_0$ has zero Maslov class, then so does $F_t$, for any $t>0$. Hence, in this case, the Legendrian angle $\theta_t$ is a well-defined $\R$-valued function and \eqref{E-zero} is well-posed. 

Assuming we have a short time smooth solution to \eqref{E-zero}, we need to show that $F_t$ satisfies \eqref{E-lmcf}. Since we know that the Legendrian condition is preserved by such a deformation, we only need to show that $\p_tF_t$ has the right component along $\xi$, namely $\eta(\p_t F_t) = 2  \theta$. By a direct computation, we have $\p_t I_f=(\p_t f, 0, -\p_t df_t)$. Since $\Psi^*\eta=ds+\beta_0$, it follows that $$\p_t F_t =\eta(\p_t \Psi\circ I_f)=\Psi^*(\eta)(\p_t I_f)= \p_t f_t= 2 \theta.$$ \end{proof}

\begin{theorem} Let $F_0: L\rightarrow S$ be a compact Legendrian submanifold with zero Maslov class in a Sasaki-Einstein manifold $S$.  Then there exists $T>0$ and a family $F_t$, $t\in [0, T)$, of Legendrian submanifolds that solves \eqref{E-lmcf} with initial condition $F_0$. 
\end{theorem}

The proof follows from adapting the results of Behrndt \cite[Sections 5.3-5.4]{Behrndt} to (\ref{E-zero}) with initial condition $f|_{t=0} = 0$.

\bibliographystyle{amsalpha}
\bibliography{lcsf}

\end{document}